\documentclass{amsart}
\usepackage[utf8]{inputenc}
\usepackage{amsmath}
\usepackage{graphicx}
\usepackage{caption}
\usepackage{xcolor}
\usepackage{subfiles}
\usepackage{comment}
\usepackage{mathrsfs}
\usepackage{enumitem}
\usepackage{dsfont}
\usepackage{tikz,tikz-cd}

\newcommand{\de}{d_{\mathrm{Eucl}}}

\newtheorem{thm}{Theorem}[section]
\newtheorem{prop}[thm]{Proposition}
\newtheorem{lemma}[thm]{Lemma}
\newtheorem{cor}[thm]{Corollary}
\theoremstyle{remark}
\newtheorem*{remark}{Remark}

\numberwithin{equation}{section}

\theoremstyle{definition}
\newtheorem{defn}[thm]{Definition}

\usepackage{mathtools}

\usepackage{hyperref}

\def\semicolon{;}
\def\applytolist#1{
    \expandafter\def\csname multi#1\endcsname##1{
        \def\multiack{##1}\ifx\multiack\semicolon
            \def\next{\relax}
        \else
            \csname #1\endcsname{##1}
            \def\next{\csname multi#1\endcsname}
        \fi
        \next}
    \csname multi#1\endcsname}

\def\calc#1{\expandafter\def\csname c#1\endcsname{{\mathcal #1}}}
\applytolist{calc}QWERTYUIOPLKJHGFDSAZXCVBNM;
\def\bbc#1{\expandafter\def\csname bb#1\endcsname{{\mathbb #1}}}
\applytolist{bbc}QWERTYUIOPLKJHGFDSAZXCVBNM;
\def\bfc#1{\expandafter\def\csname bf#1\endcsname{{\mathbf #1}}}
\applytolist{bfc}QWERTYUIOPLKJHGFDSAZXCVBNM;
\def\sfc#1{\expandafter\def\csname s#1\endcsname{{\sf #1}}}
\applytolist{sfc}QWERTYUIOPLKJHGFDSAZXCVBNM;
\def\fc#1{\expandafter\def\csname f#1\endcsname{{\mathfrak #1}}}
\applytolist{fc}QWERTYUIOPLKJHGFDSAZXCVBNM;

\newcommand{\RS}{\widehat{\mathbb{C}}}

\title[Equilibrium States for Parabolic Rational Maps]{The specification approach to Equilibrium States for Parabolic Rational Maps}
\author{Katelynn Huneycutt and Daniel J. Thompson}
\thanks{This work was partially supported by NSF grant DMS-2349915.}
\address{D.~J.~Thompson, Department of Mathematics, The Ohio State University, Columbus, OH 43210, \emph{E-mail address:} \tt{thompson@math.osu.edu}}
\address{K. Huneycutt, Department of Mathematics, The Ohio State University, Columbus, OH 43210, \emph{E-mail address:} \tt{huneycutt.13@buckeyemail.osu.edu}}
\begin{document}


\begin{abstract}
 We develop the specification and orbit-decomposition approach to equilibrium states for parabolic rational maps of the Riemann Sphere. Our result extends the well-known results on uniqueness of equilibrium states in this setting, notably the results of Denker, Przytycki and Urba\'nski. We extend the class of potentials from H\"older to those with the Bowen property on `good orbits' . We obtain uniqueness of the equilibrium state for potentials satisfying a pressure gap condition which is sharp in the class of potentials we consider.  We show that our equilibrium state has the $K$-property, and in particular it has positive entropy. When the potential is H\"older, the theory of equilibrium states is already highly developed. Nevertheless, several interesting results on equilibrium states for H\"older potentials follow readily from our approach. In the family of geometric potentials, we obtain a simple proof of uniqueness of equilibrium states up to the phase transition that occurs at the Hausdorff dimension of the Julia set. For H\"older potentials on parabolic rational maps, we show that hyperbolicity of the potential is equivalent to having a unique equilibrium state which is fully supported.  This does not appear to have been stated in the literature before, although it may be considered folklore. 
  \end{abstract}
\maketitle

\section{Introduction}

We develop the specification and orbit-decomposition approach to equilibrium states in the setting of parabolic rational maps.  Recent literature on the specification approach is split into papers which build general theory for equilibrium states \cite{CT14, Call22}, and papers which apply the general theory to obtain results for specific classes of dynamical systems using arguments and structures specialized to these systems \cite{BCFT, CESW, PYY}. This paper falls into the second category and is focused on obtaining refined uniqueness of equilbrium states results for parabolic rational maps. A rational map $f:\RS\rightarrow \RS$ is an analytic endomorphism of the Riemann sphere, $\RS$. For $z\in \bbC$, $f$ has the form $f(z) = P(z)/Q(z)$ where $P(z)$ and $Q(z)$ are nonzero relatively prime polynomials. For a rational map $f:\RS\rightarrow \RS$, we denote its Julia set by $J(f)$. We consider \emph{parabolic} rational maps $f$. These are the rational maps which have no critical points in the Julia set and have at least one rationally indifferent periodic point, i.e. a point $p$ such that $f^{n}(p) = p$ for some $n>0$ and $(f^{n})'(p)$ is a root of unity.  Let $\Omega$ denote the set of rationally indifferent periodic points for $f$. The set $\Omega$ is finite and contained in $J(f)$ and is in the closure of the forward orbit of a critical point. The system $(J(f), f)$ is a well-known example of a non-uniformly hyperbolic dynamical system. 

Thermodynamic formalism for rational maps was first developed for the measure of maximal entropy in papers by Freire, Lopes, Lyubich and Ma\~n\'e \cite{FLM, Lyubich, Mane}. Equilibrium states and related ergodic theory of this setting were developed in a series of papers by Denker, Przytycki, Urba\'nski and co-authors, notably in \cite{DU-ES, DPU} with an important improvement contributed by Inoquio-Renteria and Rivera-Letelier in \cite{IR}. The landmark result is that if $f$ is a rational map with $\deg(f)\geq 2$ and $\varphi: J(f) \rightarrow \bbR$ is a hyperbolic H\"older continuous potential, then there exists a unique equilibrium state.   Here, we say that $\varphi$ is a hyperbolic potential if there exists $n \geq 1$ such that
 \begin{equation} \label{hyperbolicpotential}
 \sup_J \frac1n S_n \varphi := \sup_{z\in J} \frac1n S_n \varphi (z) < P(\varphi),
 \end{equation}
 where $P(\varphi)$ is the topological pressure of the potential $\varphi$ on the compact dynamical system $(J(f),f)$. In the seminal result of \cite{DPU}, uniqueness of equilibrium states was proved using the a priori stronger hypothesis that $\sup_J \varphi := \sup_{z\in J} \varphi(z)< P(\varphi)$. By \cite[Proposition 3.1]{IR}, if $\varphi$ satisfies \eqref{hyperbolicpotential} there is a cohomologous function $\tilde \varphi$ which satisfies  $\sup_J \tilde \varphi < P(\varphi)$.  It thus follows from \cite{IR} that a hyperbolic H\"older potential has a unique equilibrium state. 
 
Recall that a potential is hyperbolic if and only if all equilibrium states for the potential have positive entropy.  The failure of uniqueness is often associated with the presence of a zero entropy equilibrium state at a phase transition. While hyperbolicity is a mild and natural condition, it is not perfect from the point of view of the uniqueness theory since it a priori rules out the likely source of non-uniqueness which is the possibility of zero entropy equilibrium states. The room for growth in this theory is to find natural criteria for a potential to have a unique equilibrium state without assuming that the potential is hyperbolic or to establish weaker criterion to verify hyperbolicity. 

In this paper, in the setting of parabolic rational maps, we replace the lower bound in the pressure gap \eqref{hyperbolicpotential} with the quantity
\begin{align*}
    A(\Omega, \varphi) \coloneqq \max \left\{\frac{1}{n}\sum_{i=0}^{n-1} \varphi (f^iw): w \in \Omega, f^n w = w \right\},
\end{align*}
which is the maximal average of $\varphi$ along any rationally indifferent periodic orbit. Our focus is on H\"older continuous potentials, with the H\"older class given by any smooth metric on $J(f)$.
We prove the following result.
\begin{thm}\label{thm:mainA}
    Let $f:\RS\rightarrow \RS$ be a parabolic rational map of degree $d\geq 2$ and $\Omega = \{w_{1},\dots,w_{k}\}$ be the set of its rationally indifferent periodic points.
        If $\varphi: J(f) \to \bbR$  is H\"older and $A(\Omega, \varphi)< P(J(f),\varphi)$,   then $\varphi$  has a unique equilibrium state w.r.t. the dynamical system $(J(f),f)$ and it is fully supported on $J(f)$.
\end{thm}

 Theorem \ref{thm:mainA} holds when the H\"older assumption on $\varphi$ is weakened to the technical dynamical regularity condition that $\varphi$ is `Bowen on good orbit segments'. 
 We define this property in \S\ref{background} and we verify that the `Bowen on good' property holds for H\"older potentials in \S \ref{bowenProperty}. We state the more general version of our result as Theorem \ref{generalthm}.
 
 In our setting, the pressure gap $A(\Omega, \varphi)< P(\varphi)$ is a sharp criterion to have a unique equilibrium state which is fully supported. This is because each periodic orbit in $\Omega$ corresponds to an invariant measure, so if the pressure gap fails, then at least one of these measures is an equilibrium state.   We emphasize that the `obstruction to uniqueness' term  $A(\Omega, \varphi)$ in our pressure gap is local, whereas the corresponding term in the hyperbolicity condition \eqref{hyperbolicpotential} is a global quantity. 

As an application of Theorem \ref{thm:mainA}, we consider the family of geometric potentials 
\[
\varphi_t(z) \coloneqq -t\log|f'(z)|,
\] 
which are H\"older continuous on $J(f)$. Restricted to this family, our results are not new, but our proof is simple and instructive. Let $h = \text{HD}(J(f))$ denote the Hausdorff dimension of $J(f)$. Denker and Urba\'nski \cite{DU91} showed that $h$ is the unique zero of the pressure function $t \mapsto P(\varphi_t)$, so $P(\varphi_t) > 0$ for all $t < h$. The average of  $\varphi_t$ over any parabolic cycle is zero, so $A(\varphi_{t}) = 0$. Thus, the inequality $A(\varphi_{t}) < P(\varphi_t)$ holds and the equilibrium state for $\varphi_t$ is unique for $t<h$. 
\begin{cor} \label{cor}
    Let $f$ be a parabolic rational map and $\varphi_{t}(z) \coloneqq -t\log|f'(z)|$ be the family of geometric potentials. Let $h$ be the Hausdorff dimension of the Julia set, then there exists a unique equilibrium state for all $t<h$.
\end{cor}
We could not find an explicit statement of the corollary in the literature, but we emphasize that it can be obtained readily from a criteria of Inoquio-Renteria and Rivera-Letelier which we discuss in \S \ref{geo.potential}. 
We now investigate further properties of the equilibrium state. Our key result is the following.

\begin{thm}\label{thm:posentropy}
Assume the setting of Theorem \ref{thm:mainA}. The unique equilibrium state has the $K$-property. In particular, the equilibrium state has positive entropy.
\end{thm}

For a H\"older potential, this result has strong consequences since positive entropy implies that the potential is hyperbolic. Via the results of \cite{IR}, we can then apply the results of Szostakiewicz, Urbanski and Zdunik \cite{SUZ} to obtain the following.

\begin{cor} \label{cor2}
 Let $f:\RS\rightarrow \RS$ be a parabolic rational map of degree $d\geq 2$ and $\Omega = \{w_{1},\dots,w_{k}\}$ be the set of its rationally indifferent periodic points. Let $\varphi: J(f) \rightarrow \bbR$ have the Bowen property on $\cG(\eta)$ for every $\eta>0$. The following are equivalent:
 \begin{enumerate}
 \item $\varphi$ is hyperbolic: there exists $n$ so that  $\sup_J \frac{1}{n}S_n \varphi  < P(\varphi)$
 \item $A(\Omega, \varphi)< P(\varphi)$
 \end{enumerate}
 In particular, if $\varphi$ is H\"older and $A(\Omega, \varphi)< P(\varphi)$, then the unique equilibrium state satisfies the strong statistical properties provided by \cite{SUZ}: exponential decay of correlations, central limit theorem and law of iterated logarithms.
\end{cor}
From (2), if $\varphi$ is not hyperbolic, we must have $A(\Omega, \varphi)= P(\varphi)$ and thus there is an equilibrium state which is a periodic orbit measure in $\Omega$. Thus, the hyperbolicity property is not only sufficient for, but \emph{equivalent to}, the property that there is a unique equilibrium state and the equilibrium state is fully supported. Although this result may be folklore, we emphasize that it follows readily from our approach and has not previously appeared in the literature\footnote{After completing this work, Juan Rivera-Letelier shared with us an elegant unpublished argument to show that $A(\Omega, \varphi)< P(\varphi)$ implies hyperbolicity of the potential based on direct analysis of Lyapunov exponents and facts from complex analysis. Analogous results have been proved for maps of the interval using different techniques, see  e.g.  \cite[Appendix A]{CRL25}, \cite[Corollary 3.8]{BK98}.}. 
\begin{cor} \label{cor3}
Let $f:\RS\rightarrow \RS$ be a parabolic rational map of degree $d\geq 2$. Let $\varphi: J(f) \rightarrow \bbR$ be H\"older continuous. Then there is a unique equilibrium state $\mu$, with $\mu$ fully supported on $J(f)$, if and only if $f$ is hyperbolic.
\end{cor}

Our proof of Theorem \ref{thm:mainA} is based on the specification and orbit-decomposition approach to thermodynamic formalism, and is obtained by verifying criteria developed by Climenhaga and the second-named author. These techniques have previously been applied mostly in invertible or continuous-time settings, but the general machinery applies equally to endomorphisms, as established in the original papers in this direction, e.g. \cite{CT14}. Theorem B follows from work of Call and the second-named author and follows from lifting the orbit decomposition to the product system $(J(f) \times J(f), f \times f)$. 

To the best of our knowledge, this is the first time that a specification-based approach to equilibrium states has been developed in a complex dynamics context\footnote{We note that specification ideas were used in \cite[Appendix A]{PRLS03} to connect the topological Collet-Eckman condition and uniform hyperbolicity on periodic orbits. It was observed there that these ideas have independent interest.}.  Our proof thus has methodological interest, and it is intended to lay the foundations for further exploration of specification techniques in complex dynamics. 
Existing results are usually based on the powerful and well-developed machinery of transfer operators and inducing schemes, which have had great success in this area. We anticipate that further development of the specification approach in complex dynamics will provide a complementary toolkit which has some advantages.

We briefly discuss further literature surrounding this area. There has been a recent resurgence in interest in thermodyamic formalism and complex dynamics in papers such as \cite{BianchiDinh,BianchiDinh2, BianchiHE, BH25}. In addition to the key references for general rational maps already discussed, particularly  \cite{DPU, SUZ}, the papers \cite{Przytycki, Aaronson, DU91, GPR} extensively develop the ergodic theory of the parabolic case, focusing on the natural measure and its dimension theory.  A wide class of non-uniformly hyperbolic rational maps, mostly without parabolics, were studied by Przytycki and Rivera-Letelier  \cite{PRL}. There is an extensive literature on thermodynamic formalism in complex dynamics including \cite{Briend, UV, P-survery, PRL, UrbanskiZdunki, Zinsmeister}. We particularly recommend the textbook \cite{URM} for a thorough treatment of advanced topics in this area. 

There has been extensive progress in recent years in the general theory of equilibrium states. In addition to the framework provided by specification techniques, there has been remarkable progress in smooth settings based on countable-state symbolic dynamics. The seminal work of Buzzi, Crovisier and Sarig \cite{BCS22} implies that a hyperbolic potential for a smooth diffeomorphism of a surface has a unique equilibrium state. This theory has been developed for general smooth endomorphisms in \cite{ALP, LOP}. The coding is constructed to capture typical points for all measures with exponent above some fixed threshold, and this can be used to show that there is at most one \emph{hyperbolic} equilibrium state. To use such codings for uniqueness results among \emph{all} equilibrium states, the possibility of zero exponent equilibrium states must be excluded a priori. Thus, the symbolic techniques do not yield new information for H\"older potentials for rational maps.

The paper is structured as follows. In \S \ref{background},  we develop the relevant background in thermodynamic formalism and complex dynamics. In \S \ref{mainresults}, we prove our main theorem, Theorem \ref{thm:mainA}. In \S \ref{geo.potential}, we discuss the application to geometric potentials. In \S \ref{pos.entropy}, we prove Theorem \ref{thm:posentropy} and Corollary \ref{cor2}.
\section{Background} \label{background}
\subsection{Thermodynamic Formalism} \label{tfbackground}

Given a compact metric space $(X,d)$, a continuous map $T:X\rightarrow X$, and a continuous potential function $\varphi :X\rightarrow \mathbb{R}$, an \textit{equilibrium state} for $(\varphi, T)$ is the measure that achieves the supremum, 
\begin{align*}
    P(\varphi; T) =P(\varphi) = \sup_{\mu\in M(X,T)}\left\{h_{\mu} (T) +\int \varphi \ d\mu \right\},
\end{align*}
     where $M(X,T)$ is the set of Borel $f$-invariant probability measures on $X$. For $n\in\mathbb{N}$, the $n$th Birkhoff sum is $S_n\varphi(x) \coloneqq\sum_{k=0}^{n-1}\varphi(T^{k}x)$. We have $h_{\mu}(T^n) = n h_{\mu}(T)$ and $P(S_n\varphi; T^n) = n P(\varphi, T)$. It follows that $\mu$ is an equilibrium state for $(\varphi, T)$ if and only if it is an equilibrium state for $(S_n\varphi, T^n)$.

    Bowen \cite{B74} showed if $T:X\rightarrow X$ an expansive homeomorphism of a compact metric space with the specification property, then each $\varphi:X\rightarrow \bbR$ with the Bowen property has a unique equilibrium state. In \cite{CT13}, this theorem was generalized using \textit{decompositions} of orbit into `good' and `bad' orbit segments. In this section, we give the background to state the general criteria for uniqueness that we apply to the setting of rational maps.  Given $n\in \mathbb{N}$ and $x,y\in X$, we define the $n^{\mathrm{th}}$ \textit{Bowen metric} as
\begin{align*}
    d_n(x,y) \coloneqq \sup \{d(T^kx,T^ky)\mid 0\leq k<n\}.
\end{align*}
The \textit{Bowen ball} of order $n$ and radius $\delta$ centered at $x\in X$ is defined as
\begin{align*}
    B_n(x,\delta) \coloneqq \{y\in X \mid d_{n}(x,y)<\delta \}.
\end{align*}

\begin{defn}
    A dynamical system is  \emph{positively expansive} if there exists $\beta>0$ so that for all $x\in X$, if $y\in B_{n}(x,\beta)$ for all $n\in \bbN$, then $x=y$. We say that $\beta$ is an expansivity constant.
\end{defn}

We denote an orbit segment $(x,Tx,\dots, T^{n-1}x)$ by the ordered pair $(x,n) \in X \times \bbN$. The following definitions of the specification property, Bowen potentials, and pressure are stated for orbit segments. Let $\cD \subset X\times \bbN$ be a collection of orbit segments. The orbit segments in $\cD$ of length $n$ are denoted
\[D_{n}\coloneqq \{ x\in X \mid (x,n)\in \cD\}.\]
We say the point $y\in X$ $\delta$-shadows the orbit segment $(x,n)$ if $y\in B_{n}(x,\delta)$. The specification property at scale $\delta$ allows us to find a point that $\delta$-shadows finitely many orbit segments with a uniform transition time between each orbit segment.   
\begin{defn}
A collection of orbit segments $\cD\subset X\times \bbN$ has \textit{specification at scale $\delta>0$} if there exists $\tau \in \bbN$ such that for every $(x_1,n_1),\dots, (x_{k},n_{k})\in \cD$ there exists $y\in X$ such that 
\begin{align*}
    T^{\sum_{j=1}^{i-1}(n_{j}+\tau)}(y)\in B_{n_i}(x_i,\delta) \  \text{for all} \ 1\leq i\leq k
\end{align*}
\end{defn}
A collection of orbit segments $\cD$ has the Bowen property at scale $\varepsilon$, if for any orbit segment $(x,n)\in \cD$ and any $y$ that $\varepsilon$-shadows $(x,n)$, the difference between the $n$th Birkhoff sums of $x$ and $y$ to be bounded by a uniform constant independent of $x$.
\begin{defn}
    A continuous function $\varphi : X\rightarrow \bbR$ has the \textit{Bowen property} on a collection of orbit segments $\cD \subset X\times \bbN$  if there exists a scale $\varepsilon >0$ and a constant $V>0$ such that for all $(x,n)\in \cD $ and $y\in B_{n}(x,\varepsilon)$, we have $|S_{n}\varphi(x)-S_{n}\varphi(y)|\leq V$.
\end{defn}
We now define the topological pressure of a set of orbit segments. First, for $\varepsilon>0$ and $n\in\bbN$, a set $E\subset X$ is an $(n,\varepsilon)$-\textit{separated} set if for all $x,y\in E$ with $x\neq y$, we have $d_{n}(x,y)\geq \varepsilon$.
\begin{defn}
    Given a continuous potential function $\varphi:X\rightarrow \mathbb{R}$ and a collection of orbit segments $\mathcal{D}\subset X\times \mathbb{N}$, for each $\varepsilon>0$ and $n\in\mathbb{N}$, we consider the partition sum
    \begin{align*}
        \Lambda(\mathcal{D},\varphi, \varepsilon,n) \coloneqq \sup \left\{\sum_{x\in E} e^{S_n\varphi(x)}:E\subset \mathcal{D}_{n} \ \text{is $(n,\varepsilon)$-separated} \right\}.
    \end{align*} 
The \textit{pressure} of $\varphi$ on the collection $\mathcal{D}$ at scale $\varepsilon>0$ is 
    \begin{align*}
        P(\mathcal{D},\varphi, \varepsilon)\coloneqq \limsup_{n\rightarrow \infty }\frac{1}{n}\log \left(\Lambda(\mathcal{D},\varphi, \varepsilon,n) \right) 
    \end{align*}
    and the pressure of $\varphi$ on the collection $\mathcal{D}$ is 
     \begin{align*}
        P(\mathcal{D},\varphi)\coloneqq \lim_{\varepsilon\rightarrow 0 }  P(\mathcal{D},\varphi, \varepsilon)
    \end{align*}
    When $\cD = X\times\mathbb{N}$, we obtain the topological pressure of $\varphi$ for the system $(X,T)$, see \cite{Walters}, and denote it $P(\varphi)$.
    For $\varphi \equiv 0$, $P(\cD,0)$ is the topological entropy of $T$, denoted $h_{\text{top}}(T)$. 
\end{defn}

\begin{defn} 
    A one-sided \textit{decomposition} of $X\times \bbN$ consists of two collection $\cG,\cS\subset X\times \bbN$ for which there exist two functions $g,s:X\times \bbN\rightarrow \bbN$ such that for every $(x,n)\in X\times \bbN$, the values $g=g(x,n)$ and $s=s(x,n)$ satisfy $g+s=n$ and
    \begin{align*}
       (x,g)\in \cG, \quad (T^{g}x,s)\in \cS
    \end{align*}
\end{defn}

Loosely speaking, we split the orbit segment into a `good' segment in $\cG$ and a `suffix' in $\cS$. On a `good' segment, we will require $T$ to have the specification property and $\varphi$ to have the Bowen property. The pressure of the `suffix' set should be less than the pressure on the whole space. The following abstract theorem is our tool for verifying uniqueness of equilibrium states.
 \begin{thm}[Climenhaga-Thompson] Let $X$ be a compact metric space, $T:X\rightarrow X$ a positively expansive continuous map, and $\varphi:X\rightarrow \mathbb{R}$ a continuous potential function. Suppose 
$X\times \mathbb{N}$ admits a decomposition $(\cG,\cS)$ with the following properties: 
\begin{enumerate}
    \item $\mathcal{G}$ has specification at any scale $\delta >0 $ 
    \item $\varphi$ has the Bowen property on $\mathcal{G}$
    \item $P( \mathcal{S},\varphi)<P(\varphi)$
\end{enumerate}
Then $(X,T,\varphi)$ has a unique equilibrium state $\mu$. The measure $\mu$ is fully supported. \label{thm: CT16}
\end{thm}
The above result was proved by Climenhaga and the second named author in \cite[Theorem 5.5]{CT16} for homeomorphisms with a non-uniform expansivity assumption. The non-invertible case with a non-uniform positive expansivity assumption is essentially the same proof, with some simpler book-keeping.  In the measure of maximal entropy case, the non-invertible case was treated with full details in \cite{CT14}. 

Although the details of the equilibrium state proof have not been written out in the non-invertible setting, it is clear that the proof of \cite{CT16} adapts without issue, and the proof in \cite{CT14} for the measure of maximal entropy can be used as a roadmap for the necessary adaptations. 

\subsection{Dynamics of rational maps}

We recall some preliminaries from complex dynamics. For a more detailed treatment of the this topic see \cite{Blanchard, Beardon, Milnor}.
Let $f:\RS\rightarrow \RS$ be a rational map (an analytic self-map of the Riemann sphere). The \textit{Fatou set} of $f$, denoted $F(f)$, is the maximal open subset of $\RS$ on which $\cF = \{f^{n}: n\in\bbN \}$ is normal (or equicontinuous). The \textit{Julia set} of $f$, $J(f)$, is the complement of the Fatou set. $J(f)$ is compact, perfect, and completely invariant under $f$, that is $J(f) = f^{-1}(J(f))=f(J(f))$. 

One key tool in this setting is Montel's Theorem, which states if each $f\in\cF$ fails to take on 3 points in $\RS$, then $\cF$ is a normal family. The next two lemmas are consequences of this Theorem. First, we have the normality of inverse branches. 
\begin{lemma}
    Let $f$ be a rational map with $\text{deg}(f)\geq 2$, and suppose that the family $\{f_\nu^{-m}:m\geq 1\}_{m,\nu}$ is such that each $f^{-m}_{\nu}$ is a single-valued analytic branch of some $(f^m)^{-1}$ in domain $D$. Then $\{f_\nu^{-m}:m\geq 1\}_{m,\nu}$ is normal in $D$. \label{normalInverseBranches}
\end{lemma}

Recall a map $f$ is \textit{locally eventually onto}, or topologically exact, on a set $X$ if for every nonempty open set $U$, there exists $n\in\bbN$ such that $f^{n}(U)\supset X$. On any neighborhood $U$ of  $z\in J(f)$,  the family $\cF$ is not normal. By Montel's Theorem, it follows that 
$$J\subset \bigcup_{n=0}^{\infty} f^{n}(U).$$
Using this fact and compactness of $J(f)$, one can show that every rational map is locally eventually onto its Julia set.

\begin{lemma}
    The rational map $f:J(f)\rightarrow J(f)$ is locally eventually onto. \label{LEO}
\end{lemma}

\subsection{Hyperbolicity and the hyperbolic metric}
Let $(Y, d)$ be a compact metric space and $f:Y \to Y$ be continuous. We call $f$ \textit{(locally) uniformly distance-increasing} on $(Y,d)$ if there exists an $\epsilon>0$ and $L >1$ such that 
\begin{align*}
d(fx ,fy) \geq Ld(x,y)
\end{align*}
whenever $x,y \in Y$ and $d(x,y)<\epsilon$. We call $L$ the \textit{expansion constant}. We say $f$ is \textit{(locally) distance-increasing} on $(Y, d)$ if there exists an $\epsilon>0$ such that 
\begin{align*}
d(fx ,fy) > d(x,y)
\end{align*}
whenever $x,y \in Y$ and $d(x,y)<\epsilon$.
Let $M$ be a compact differentiable manifold without boundary and $f:M\rightarrow M$ be a continuous map. Let $X \subseteq M$ be $f$-invariant and suppose that the restriction of $f$ to $X$ is a $C^{1}$-endomorphism. We say that $f$ is \textit{hyperbolic} on $X$ if for any Riemannian metric, $g$, there exists a $c>0$ and $L>1$ such that 
\begin{align*}
    \|Df^{n}v\|_{g}\geq c L^{n}\|v\|_{g} \quad \text{for all $v\in TX$ and $n>0$.}
\end{align*}
Let $d_{g}$ be the Riemannian distance function. If $X$ is compact and $f$ is hyperbolic on $X$ with respect to a Riemannian metric $g$, it follows that there exists $m\in \mathbb{N}$ such that $f^{m}$ is uniformly distance-increasing on $(X, d_g)$.

Suppose now that $f$ is a rational map. Let $C(f)$ denote the set of critical points of $f$, then the \textit{postcritical set} of $f$ is
\begin{align}
    P(f) \coloneqq \overline{\bigcup_{c\in C(f),n>0} f^{(n)}(c)}.
\end{align}
When $P(f)\cap J(f) =\emptyset$, then for any Riemannian metric $f:J(f)\rightarrow J(f)$ is hyperbolic, and there exists an $m>0$ such that $f^{m}$ is uniformly distance-increasing. If we want to ensure $f$ itself is uniformly distance-increasing, we need the \textit{hyperbolic metric for f}.

Let $\mathbb{D}$ be the unit disk. The \textit{hyperbolic metric} on $\mathbb{D}$ is the unique metric, up to multiplication by a constant, that is invariant under every conformal automorphism of $\mathbb{D}$. A subset $\Delta \subset \mathbb{C}$ is called a \textit{hyperbolic region} if $\RS \setminus \Delta$ contains at least three points. Given a hyperbolic region $\Delta$, there exists a unique metric on $\Delta$ whose pullback by any holomorphic universal covering map $h: \mathbb{D} \to \Delta$ is the hyperbolic metric on $\mathbb{D}$. This metric is called the \textit{hyperbolic metric} on $\Delta$, denoted $\rho_{\Delta}(z)|dz|$. 

Whenever $P(f)$ contains at least three points, we can define the \textit{hyperbolic metric for $f$} to be the hyperbolic metric on the hyperbolic region $\RS \setminus P(f)$ and we write $\rho_H:= \rho_{\RS \setminus P(f)}$. As a consequence of the generalized Schwarz-Pick Lemma, a rational map $f$ is expanding on $f^{-1}(\RS \setminus P(f))$ with respect to this metric.
\begin{thm} 
  Let $f$ be a rational map of degree $d\geq 2$ such that $J(f)\cap P(f)=\emptyset$. Then  $\rho_{H}|dz|$ is a uniformly expanding metric on $J(f)$.
\end{thm}
Parabolic rational maps do not satisfy the condition of the above theorem because the set of rationally indifferent points $\Omega\neq \emptyset$ and $\Omega\subset P(f)\cap J(f)$. The function $\rho_H$ blows up at $P(f)$ and thus does not define a metric on all of $J(f)$. Although we do not have a uniformly expanding metric in the parabolic case, we do have the positive expansivity property.
\begin{thm}[\cite{DU91}, Theorem 4] \label{thm:expansive} 
 Let $f$ be a rational map of degree $d\geq 2$ such that $J(f)\cap C(f)=\emptyset$. Then $f$ is positively expansive.
\end{thm}

Thus, the class of parabolic rational maps is characterized by the following corollary.
\begin{cor}[\cite{DU91}, Corollary 5]
    A positively expansive rational map $f$ is not uniformly expanding on $J(f)$ if and only if $\Omega \neq \emptyset$.
\end{cor}

By Sullivan's no wandering domain theorem and the classification of Fatou components, if an orbit of a point in the Fatou set accumulates on $J(f)$, then it accumulates on a rationally indifferent periodic point. Denker and Urba\'nski \cite{DU91} restated this fact in the next lemma which they proved in a fixed Riemannian metric $d_g$. The statement trivially extends to any metric which is uniformly equivalent to $d_g$. We write $B(\Omega,\alpha; d) = \bigcup_{p \in \Omega} B(p, \alpha; d)$.

\begin{lemma}[\cite{DU91}, Lemma 2] \label{lem:inversebranches}
  Let $d$ be a metric on $J(f)$ which is equivalent to a Riemannian distance function on $J(f)$. If $J(f)$ contains no critical point of $f$ then for every $\alpha>0$ there exists $\delta=\delta(\alpha)>0$ such that for every $z\in J(f)\setminus B(\Omega,\alpha; d)$, the ball $B(z,2\delta; d)$ does not intersect
the forward trajectory of any critical point of $f$. In particular all analytic inverse branches of $f^{n}:\RS\rightarrow \RS$ are well defined on $B(z,2\delta; d)$ for every $n \in \mathbb N$. \label{inverseBranches}
\end{lemma} 

\subsection{Metric expansion properties and the Milnor metric} \label{s.milnor} We now discuss metric expansion properties, and we refer the reader to \cite[Appendix B]{milnor2critical} for more details. First, we discuss some natural metrics.  For parabolic rational maps, the Fatou set is always non-empty, so we can analytically conjugate $f$ to a map $g$ whose Julia set $J(g)$ is contained in $\mathbb{C}$. Then the Euclidean metric is a smooth metric on $J(g)$.   For a parabolic rational map, the hyperbolic metric is only well-defined and expanding on
$f^{-1}(\RS \setminus P(f))$. We thus cannot use the hyperbolic metric for an expanding metric at the parabolic points or their preimages. We could consider the Euclidean metric on $J(f)$. Since $J(f)$ contains no critical points,  we know that there exists $c>0$ so that $|f'(z)| \geq c >0$ for all $z \in J$, but we do not have any better global lower bound. However, as a consequence of the Leau-Fatou Flower Theorem, in a  punctured neighborhood of a parabolic fixed point the Euclidean derivative is strictly greater than one.  Milnor proved that a parabolic rational map $f$ admits a metric such that $f$ is distance-increasing on $J(f)$. The idea is to combine the desirable properties of the hyperbolic and Euclidean metrics. Milnor's metric is defined piecewise, as a rescaling of the Euclidean metric near the rationally indifferent periodic points and the hyperbolic metric away from them. This construction is the key point which proves the following lemma.

\begin{lemma}[\cite{milnor2critical}, Lemma B.3]
    Let $f$ be a rational map restricted to its Julia set J(f). There exists a metric $d$ so that $(J(f), d)$ is distance-increasing if and only if f has no critical points in $J(f)$.
\end{lemma}

We now define an adaptation of the Milnor metric such that $f$ is not only distance-increasing, but also is uniformly distance-increasing on $J(f)$ for points $z$ which are away from the parabolic points.  We assume that all points in $\Omega$ are fixed points. If this is not the case, we can pass to a suitable iterate $f^n$.  Milnor's definition involves a choice of constant $M$, and the only change in our version of the Milnor metric is that we need to choose this constant more carefully.

For a metric $\rho(z) |dz|$, let $\|f'(z)\|_\rho$ denote the derivative computed in this metric. the relationship between the derivative in $|dz|$ and in $\rho(z) |dz|$ is
\begin{align} \label{derivativeformula}
    \|f'(z)\|_{\rho}= \frac{\rho(f(z))|f'(z)|}{\rho(z)}.
\end{align}

    For a parabolic rational map $f$ with degree $d\geq 2$,  let $\rho_H(z)|dz|$ be the hyperbolic metric for $f$ on the hyperbolic region $\RS\setminus P(f)$. Let $B(\Omega, \alpha;\de)$ denote the union of Euclidean metric balls around each $\omega\in \Omega$. Let $K(\alpha) = J\setminus B(\Omega, \alpha;\de)$. Let 
    \[
    r(\alpha)= \min \{ \|f'(z)\|_{\rho_H} : z \in K(\alpha)\cap f^{-1}K(\alpha) \}.
    \] 
    On $K(\alpha)\cap f^{-1}K(\alpha)$, $\|f'(z)\|_{\rho_H}>1$, so by compactness $\|f'(z)\|_{\rho_H}$ achieves a minimum on $K\cap f^{-1}K$, and we have $r(\alpha)>1$. 
We define $M= M(\alpha)$ to be
   \[
   M  := \max\left \{\frac{r(\alpha) \rho_H(z)}{|f'(z)|}: z \in f^{-1}(\Omega)\setminus \Omega \right \} +1.
   \]
  Note that for $z$ in the finite set $f^{-1}(\Omega)\setminus \Omega$, we have 
    \begin{align}
        M>\frac{r(\alpha) \rho_H(z)}{|f'(z)|}.
    \end{align}
        For each sufficiently small $\alpha>0$, the Milnor metric, $\psi_\alpha(z)|dz|$, is defined by 
        \begin{align}
        \psi_\alpha(z) = \begin{cases}
        \rho_H(z) \quad &\text{if} \  z\not \in B(\Omega, \alpha;\de)\\
        M \quad &\text{if}\  z \in B(\Omega, \alpha;\de).
        \end{cases}
        \end{align}

Let $d_\alpha$ be the path metric determined by $\psi_\alpha(z)|dz|$. We call this the Milnor metric \footnote{If we did not assume that $\Omega(f)$ has only fixed points, we can find $g=f^k$ so $\Omega^g$ contains only fixed points and the Milnor distance is defined to be
$d^f_\alpha(x,y) = \sum_{j=0}^{k-1} d^g_{\psi_\alpha}\left(f^{j}x,f^{j}y\right)$.
}

To simplify our estimates, we proceed under the assumption that $\Omega(f)$ has only fixed points. The proof that $d_{\alpha}$ is distance-increasing on $J(f)$ is in \cite[Lemma B.3]{milnor2critical}. We adapt the proof to show that $f$ is uniformly distance-expanding (w.r.t $d_\alpha$) on $K(\alpha)$ 
follows part of Milnor's proof in \cite[Lemma B.3]{milnor2critical}. 
\begin{lemma} \label{distanceincreasing}
    For $\alpha>0$ sufficiently small, $f$ is uniformly distance-increasing on $K(\alpha)$ with respect to $d_\alpha$:
    There exist $\varepsilon>0$ so that 
    \begin{align*}
        d_\alpha(fx,fy)\geq r(\alpha) d_\alpha(x,y)
    \end{align*} whenever $x, y \in K$ and $d_\alpha(x,y)<\varepsilon$.
    \label{uniform}
\end{lemma}
\begin{proof} We write $K=K(\alpha)$ and $r=r(\alpha)$. Therefore, on $K\cap f^{-1}K$, we have 
\begin{align*}
\psi_\alpha(f(z))|f'(z)| =\rho_H(f(z))|f'(z)| \geq r \rho_H(z)=r \psi_\alpha(z).
\end{align*}
Suppose $z\in K\cap f^{-1}B(\Omega, \alpha; \de)$. Given $\alpha'$, there exists an $\alpha$ such that for all $z\in f^{-1}B(\Omega,\alpha; \de)$, there exists $\hat{z}\in f^{-1}\Omega\setminus \Omega$ with $\de(z,\hat{z})<\alpha'$. By the definition of $M$, 
\begin{align*}
    M|f'(\hat{z})|\geq r\rho_H(\hat{z}).
\end{align*}
For $\alpha$ sufficiently small, it also holds that
\begin{align*}
    M|f'(z)|\geq r\rho_H(z).
\end{align*}
and
\[
\psi_\alpha(f(z))|f'(z)| =M|f'(z)| \geq r \rho(z) = r \psi_\alpha(z).  \qedhere
\]
\end{proof}
Note that for $p \in \Omega$, $B(p, M^{-1} \alpha; \psi_\alpha(z)|dz|) = B(p, \alpha; \de)$. For all sufficiently small $\alpha>0$, we can assume without loss of generality that $(J(f), f)$ is expansive at scale $4 \alpha$ in the metric  $d_\alpha = \psi_\alpha(z)|dz|$.

To see this, we first take a reference expansivity constant $\epsilon_0$ for $(J(f), f)$ w.r.t the Euclidean metric $|dz|$. We assume that the scale $\alpha$ is chosen small enough compared to this scale. Since the metric is only changed in an $\alpha$-neighborhood, this does not affect the expansivity constant.  More precisely, for a metric $\rho |dz|$, we have that 
$\de(x, y) \leq (\inf \rho)^{-1} d_\rho (x,y)$, so if $d_\rho(x, y) \leq \epsilon$, then $\de (x, y) \leq (\inf \rho)^{-1} \epsilon$ and so if $(J(f), f)$ is expansive w.r.t $\de$ at scale $\epsilon$, then $(J(f), f)$ is expansive w.r.t $d_\rho$ at scale $(\inf \rho) \epsilon$. 
We choose a reference scale $\epsilon_0$ so that $\epsilon_0$ is an expansivity constant for $(J(f), f)$ w.r.t $\de$. Let  $\kappa = \min \{1, \inf \rho_H \}$. Since $\rho_H >0$ on $J(f)\setminus \Omega$ and $M\geq 1$, we have
\[
\inf \psi_\alpha \geq \kappa.
\]
We know that $\inf \psi_\alpha \epsilon_0$ is an expansivity constant  for $\psi_\alpha(z)|dz|$, and thus it follows that $\kappa \epsilon_0$ is an expansivity constant for any metric   $\psi_\alpha(z)|dz|$. Fix an $\alpha \in (0, \kappa\epsilon_0/4)$ and also small enough that Lemma \ref{uniform} applies. In particular, we know that $4 \alpha$ is an expansivity constant in the metric $d_\alpha$.  

For our fixed $\alpha>0$, we take the corresponding Milnor metric. We write $\psi= \psi_\alpha$ and $d=d_\alpha$. From now on, all distances are taken with respect to this Milnor distance unless we specify otherwise. Since we are interested in H\"older continuous potential functions, we check that the Milnor metric gives the same H\"older class as any smooth metric $g(z)|dz|$ on $J(f)$. 
\begin{lemma}
    The metric $\psi(z)|dz|$ is Lipschitz equivalent to any smooth metric on $J(f)$.
\end{lemma}

\begin{proof}
It suffices to show that $\psi(z)|dz|$ is Lipschitz equivalent to the Euclidean metric $|dz|$. Since $K=J\setminus B(\Omega, \alpha; \de)$ is compact, $\rho(z)\neq 0$, and $\rho(z)$ is finite away from $\Omega$, then there exists constants $c,C>0$ such that for all $z\in K$
$$c|dz| \leq \psi(z)|dz| \leq C|dz|. $$
On $B(\Omega, \alpha; \de)$, $\psi(z) \equiv M$. Thus, for any $z\in J(f)$, 
\[
\min\{c, M \}|dz| \leq \psi(z)|dz| \leq \max\{C, M \}|dz|. \qedhere
\]
\end{proof}
A consequence of this equivalence is that Lemma \ref{lem:inversebranches} holds true when the balls in the statement are taken in our Milnor distance.
\section{Proof of main result} \label{mainresults}
We now define some natural orbit decompositions of our space. Our decompositions are examples of one-sided $\lambda$-decompositions as described in \cite{Call22}. The set $\cB(\eta)$ is our set $\cS$ of `suffixes'. 
\begin{defn} \label{decomposition}
    Let $\lambda:X\rightarrow [0,\infty)$ be a bounded, lower semicontinuous function. For all $\eta\in [0,1]$, define
    \begin{align}
        \cB(\eta) &\coloneqq \left\{(x,n) |\ \frac{1}{n}\sum_{i=0}^{n-1} \lambda (f^ix)< \eta \right\} \\
        \cG(\eta) &\coloneqq\left\{(x,n) |\ \frac{1}{k}\sum_{i=n-k}^{n-1} \lambda (f^ix)\geq \eta \ \text{for}\  1\leq k\leq n \right\} \label{good}
    \end{align}
We decompose an orbit segment $(x, n)$ by taking the largest $k\leq n$ so that
\[
(f^{n-k}x,k)\in \cB(\eta), \hspace{10pt} (x,n-k)\in \cG(\eta).
\] 
We say that the decomposition $(\cG(\eta), \cB(\eta))$ obtained this way is the (one-sided) $\lambda$-decomposition of $X \times \mathbb N$.
\end{defn} 

Our choice of $\lambda$ will be the characteristic function of the complement of a neighborhood of $\Omega$, which we will choose precisely shortly. 

\subsection{Specification}
We will begin by showing a sufficient condition for $(X,T)$ to have the specification property and then show it holds for parabolic rational maps\footnote{The specification property can also be obtained as a corollary of the existence of a Markov partition for parabolic rational maps \cite{dR82, Du91Absolute} (although the projection map is not H\"older which limits the utility of the Markov partition in the parabolic case). We prefer to give a self-contained proof which shows the fundamental idea behind this phenomenon and is suitable for generalization.}.

    The proof of the specification property away from $\Omega$ is rather general. Although, the balls used in this proof are computed in our metric $d$ to match the rest of our presentation, the argument does not require the use of this metric and would hold equally for any choice of Riemannian metric on $\RS$. We write $B(z, \epsilon)$ for $B(z, \epsilon;d)$ unless we specify otherwise.

\begin{lemma}
    Let $X$ be compact metric space and $T:X\rightarrow X$ be a continuous map. Suppose that for all $\varepsilon>0$ there exists $N_{\varepsilon}\in\bbN$ such that  for all $(x,n)\in \cD\subset X\times \bbN$, we have
    \begin{align}
        T^{N_{\varepsilon}}(T^{n-1}B_{n}(x,\varepsilon))=X. \label{BowenBallCover}
    \end{align}
 Then $(X,T)$ has specification on  $\cD$ at scale $\varepsilon$. \label{suffCon}
\end{lemma}

\begin{proof}
    Consider $(x_1,n_1),\dots, (x_k,n_k)$.    Let $y_{k} = x_{k}$. By (\ref{BowenBallCover}), we can find a point $y_{k-1}\in B_{n_{k-1}}(x_{k-1},\varepsilon)$ with 
 \begin{align*}
    T^{N_{\varepsilon}+n_{k-1}-1}y_{k-1} = y_{k} 
\end{align*}
  We continue recursively, extending the orbit of $y_{k}$ backwards to get $y_{k},y_{k-1}, \dots y_{i}$. Let $y_{i-1}$ satisfy $y_{i-1}\in B_{n_{i-1}}(x_{i-1},\varepsilon)$ and $T^{N_{\varepsilon}+n_{i-1}-1}y_{i-1}=y_{i}$. Let $y\coloneqq y_1$, then $$T^{\sum_{i=1}^{j}(n_{i}-1)+jN_{\varepsilon}}y = y_{j+1}$$ for all $j\in \{1,\dots,k-1 \}$. By construction, $y_{j+1}\in B_{j+1}(x_{j+1},\varepsilon)$, so we have specification.
\end{proof}
For $\alpha>0$, we define 
\[
E(\alpha)\coloneqq J(f)\setminus \overline{B(\Omega,2\alpha)}.
\]
We show that for all $\alpha>0$, parabolic rational maps satisfy \eqref{BowenBallCover} on the set of orbit segments $\cD(\alpha)\subset J(f)\times \mathbb{N}$ given by 
\begin{align*}
    \cD(\alpha) \coloneqq \{(x,n)\in J(f)\times \mathbb{N} \mid f^{n-1}x\in E(\alpha) \}.
\end{align*}

\begin{lemma}
For a parabolic rational  map, for all $\varepsilon>0$, there exists a $\gamma(\varepsilon)$ such that for all $(x,n)\in \cD(\alpha)$, we have 
\begin{align*}
      B(f^{n-1}x,\gamma(\varepsilon))\subset f^{n-1}B_{n}(x,\varepsilon)
\end{align*}\label{defSize}
\end{lemma}
\begin{proof}
Let $x\in J\setminus B(\Omega,2\alpha)$. By Lemma \ref{inverseBranches}, there exists a $\delta (\alpha)$ such that the inverse branches of $f^{n}:\RS\rightarrow \RS$ are well defined on $B(x,2\delta)$ for every $n\in \bbN$. We can cover $J\setminus B(\Omega,2\alpha)$ by 
\begin{align*}
\bigcup_{x\in J\setminus B(\Omega,2\alpha)} B(x,2\delta).
\end{align*}

Let $\Gamma = \{ D_1,\dots,D_{m}\}$ be the finite subcover of this cover. For each $D_{i}\in \Gamma$, by Lemma \ref{inverseBranches} and  \ref{normalInverseBranches}, the family of all inverse branches on $D_{i}$ is normal, and thus also equicontinuous. Therefore, on each $D_{i}\in \Gamma$, given $\varepsilon> 0$, there exists $\gamma_i$ such that $d(f^{-j}_{v} x,f^{-j}_{v} y)< \varepsilon$ whenever $d(x,y)<\gamma_i$ for any inverse branch $f^{-j}_v$, $1\leq \nu\leq d^{j}$, of $f^{j}$.  Let $\delta'$ be the Lebesgue number for the cover $\Gamma$ and 
\begin{align*}
\gamma(\varepsilon) = \min \{\delta',\gamma_1,\dots,\gamma_{m}\}.
\end{align*}

Suppose $(x,n)\in \cD(\alpha)$ and $z\in  B(f^{n-1}x,\gamma(\varepsilon))$. By the definition of $\cD(\alpha)$, $f^{n-1}x\in J\setminus B(\Omega,2\alpha)$. Since $d(z,f^{n-1}x)<\delta'$, they are both contained in the same domain $D_k\in \Gamma$ for some $1\leq k\leq m$. Additionally, $d(z,f^{n-1}x)<\gamma_k$ for the $\gamma_{k}$ coming from equicontinuity on that domain $D_{k}$. For $1\leq j\leq n-1 $, choose the inverse branch $f^{-j}_{v}$ on $D_{k}$ given $f^{-j}_{v}f^{n-1}x = f^{n-1-j}x$. By equicontinuity, we have $d(f^{-j}_{v}z,f^{-j}_{v}f^{n-1}x)\leq \varepsilon$. Let $y = f^{-(n-1)}z$, then $y\in B_{n}(x,\varepsilon)$ and $z\in f^{n-1}B_{n}(x,\varepsilon)$.
\end{proof}

\begin{lemma}
    If a map $T:X\rightarrow X$ is locally eventually onto, then for all $\varepsilon >0 $, there exists an $N_{\varepsilon}>0$ so that for all $y\in X$, 
    $$f^{N_\varepsilon}(B(y,\varepsilon))=X$$
\end{lemma}
\begin{proof}
    Let $\Gamma$ be a finite subcover of $\{B(x,\frac{\varepsilon}{2}) : \ x\in X \}$. Let $N(x,\varepsilon)$ satisfy $f^{N(x,\varepsilon)}B(x,\varepsilon/2)$. Such an $N(x,\varepsilon)$ always exists as $T$ is locally eventually onto. Let $N_{\varepsilon} = \max\{N(x,\varepsilon) : B(x,\frac{\varepsilon}{2}) \in \Gamma\}$. If $y\in X$, then $y\in B(x,\frac{\varepsilon}{2})$ for some $B(x,\frac{\varepsilon}{2})\in \Gamma$, as $\Gamma$ covers $X$. Further, by the triangle inequality, $B(x,\frac{\varepsilon}{2})\subset B(y,\varepsilon)$. It follows that $f^{N_{\varepsilon}}(B(y,\varepsilon))\supseteq f^{N_{\varepsilon}}B(x,\frac{\varepsilon}{2}) = X$
\end{proof}
By lemmas \ref{defSize} and \ref{LEO}, we arrive at the following corollary. 
\begin{cor}
    For a parabolic rational map, for all $\varepsilon> 0 $, there exists a $N_{\varepsilon}$ such that for all $(x,n)\in \cD(\alpha)$, we have
    $$f^{N_{\varepsilon}}\left(f^{n-1}B_{n}(x,\varepsilon)\right)=J (f).$$ \label{cor: propforParabolic}
\end{cor}
 Thus, by Lemma \ref{suffCon} and Corollary \ref{cor: propforParabolic}, we have shown the parabolic rational maps satisfy the specification property on $\cD(\alpha)$ for any $\alpha>0$.
  \begin{lemma} 
 Let $f:J(f)\rightarrow J(f)$ be a parabolic rational map, then $(f, J(f))$ has specification on $\cD(\alpha)$. \label{parabolicSpecification}
 \end{lemma}

\subsection{Choice of $\lambda$-decomposition and the Bowen Property} \label{bowenProperty}

We now make precise the choice of $\lambda$ in the decomposition Definition \ref{decomposition}. In this section, we use Lemma \ref{distanceincreasing}, so for our presentation we are working with the assumption that $\Omega(f)$ contains only fixed points. At the end of \S 2, we fixed an $\alpha$ satisfying Lemma~\ref{uniform}. This fixed a Milnor metric and distance function. We have $B(\Omega, \alpha; \de) = B(\Omega, M^{-1} \alpha; d) \subset B\left(\Omega, 2\alpha; d\right) = B\left(\Omega, 2\alpha\right )$. Note that we chose $\alpha$ so that $4 \alpha$ is an expansivity constant in the metric $d$. 

Let $\lambda = \lambda^{\alpha}$ be the characteristic function 
\begin{align}
\lambda \coloneqq \chi_{E} \quad \text{where  $E=E(\alpha) \coloneqq J\setminus \overline{B(\Omega,2\alpha)}$}. \label{choiceDecomp}
\end{align} 

Let $(\cG(\eta), \cB(\eta)) = (\cG^\alpha(\eta), \cB^\alpha (\eta))$ be the $\lambda$-decomposition as in Definition \ref{decomposition} for this choice of $\lambda$. The choice of $\alpha$ is fixed as in \S \ref{s.milnor} until the end of  \S \ref{mainresults}. Let $r=r(\alpha)>1$ be as in Lemma \ref{uniform}. 
\begin{lemma}
  Let $\varepsilon>0$ be sufficiently small.  Let $(x,n)\in\cG(\eta)$ and $y\in B_{n}(x,\varepsilon)$, then 
    \begin{align*}
        d(f^{\ell}x,f^{\ell}y)\leq r^{-\eta(n-\ell)} \varepsilon
    \end{align*}
    for every $0\leq \ell\leq n$
\end{lemma}
\begin{proof}
Let $x\in E$, then  $x,y\in J\setminus \overline{ B(\Omega, 2\alpha)}\subset K(\alpha)$. By Lemma \ref{uniform}, if $x,y\in K(\alpha)$ we have
\begin{align}
    d(fx,fy)\geq  r d(x,y) \label{exp}
\end{align}
whenever $x$ and $y$ are close. Therefore, if $f^{j}x\in  E$ and $y\in B_{n}(x,\varepsilon)$, then
\begin{align*}
    d(f(f^{j}x),f(f^{j}y))\geq  r d(f^j x,f^j y).
\end{align*}

Let $0\leq \ell \leq n-1$, then there exists a $1\leq k\leq n$ such that $\ell = n-k$. Since $(x,n)\in\cG(\eta)$, for any fixed $1\leq k\leq n$, 
\begin{align*}
    \sum_{i=n-k}^{n-1} \chi_{E} (f^ix)\geq \eta k.
\end{align*}
We find that for each $\ell$, 
\begin{align}
    |\{j\mid \ell \leq j\leq n-1 \ \text{and} \ f^{j}x\in E \}|\geq \eta (n-\ell)
\end{align}
Thus, there are at least $\eta (n-\ell)$ iterates in the orbit $(f^{\ell}x,n-1-\ell)$ contained in $E$ and thus in $K(\alpha)$. Given $m\in \mathbb{N}$, if the iterate $f^{m}x\not\in E$, since the metric $d$ is distance increasing and $y\in B_{n}(x,\varepsilon)$, we have
\begin{align*}
    d(f(f^{m}x),f(f^{m}y))> d(f^m x,f^m y). 
\end{align*}

Therefore, it follows that 
\[
    d(f^{\ell}x,f^{\ell}y)\leq r^{-\eta(n-\ell)}  d(f^{n-1}x,f^{n-1}y)\leq r^{-\eta(n-\ell)} \varepsilon \qedhere
\] \
\end{proof}

\begin{lemma} \label{HoldertoBowen}
    Any H\"older continuous $\varphi$ has the Bowen property on $\cG$ at scale $\varepsilon$. \label{lem: holderBowen}
\end{lemma}

\begin{proof}
Suppose $(x,n)\in \cG(\eta)$ and $y\in B_{n}(x,\varepsilon)$. Let $\varphi: J(f)\rightarrow \mathbb{R}$ be H\"older continuous with H\"older constant $K$ and H\"older exponent $\alpha$.
Then we have 
    \begin{align*}
        |S_{n}\varphi(x)- S_{n}\varphi(y)| &\leq  \sum_{\ell=0}^{n-1} d(f^{\ell}x,f^{\ell}y)^{\alpha}\\
        &\leq K \sum_{\ell=0}^{n-1}r^{-\eta(n-\ell)\alpha} \varepsilon^{\alpha}\\
        & = K\varepsilon^{\alpha}  \sum_{k=1}^{n}\left(\frac{1}{r^{\eta\alpha}}\right)^k\\
        & \leq K \varepsilon^{\alpha} \frac{1}{1-r^{-\eta\alpha} } \qedhere
    \end{align*}
\end{proof}

\subsection{Pressure Estimate}

We adapt an argument by Call and the second-named author \cite{CallT22} to  verify the pressure gap condition $P(B(\eta), \varphi)<P(\varphi)$.
\begin{prop}[\cite{Call22}, Proposition 5.9]\label{pressureEstimate}
    For any $\lambda$-decomposition, if the entropy map is upper semicontinuous
    \begin{align}
        \lim_{\eta\downarrow 0} P(B(\eta),\varphi)\leq \sup \left\{P_{\nu}(\varphi) \big| \int \lambda\ d\nu = 0 \right\} = P(B_{\infty},\varphi) \label{pressureBound}
    \end{align}
where  $B_{\infty} = \bigcap_{n\in \bbN} f^{-n}\lambda^{-1}(0)$.
\end{prop}

Parabolic rational maps are positively expansive, so  the entropy map is upper semicontinuous. Therefore, we can proceed by computing the right hand side of (\ref{pressureBound}) in Proposition \ref{pressureEstimate}.
\begin{lemma} \label{average}
For $\lambda\coloneqq \chi_{E}$ as defined previously, then $P(B_{\infty},\varphi) = A(\Omega, \varphi) $ where
     \begin{align*}
    A(\Omega, \varphi) \coloneqq \max \left\{\frac{1}{n}\sum_{i=0}^{n-1} \varphi (f^iw): w \in \Omega, f^n w = w \right\},
\end{align*}

\end{lemma}

\begin{proof}
We use the same argument as in the proof of positive expansivity in \cite{DU91}.
First, note that
\begin{align*}
    B_{\infty} =  \bigcap_{n\in \bbN \cup \{0\}} f^{-n} \overline{B(\Omega,2\alpha)}. 
\end{align*}
If $x\in B_{\infty}$, then $d(f^{n}x,\Omega)\leq 2\alpha$ for all $n\in \bbN$. 
Since $2\alpha$ is a constant of expansivity, it is smaller than the minimum distance between distinct points of $\Omega$. Hence, there exists $w \in \Omega$ such that
$d(f^{n}x,w)\leq 2\alpha$ for all $n\in \bbN$. Recall that we assumed $f$ only has fixed points, so $d(f^{n}x,f^{n}w)\leq 2\alpha$ for all $n\in \bbN$. By expansivity,  this implies $x=w$, and thus $B_\infty = \Omega$. It follows that  $P(B_{\infty},\varphi) = P(\Omega,\varphi) = A(\Omega, \varphi)$.
\end{proof}

\subsection{Completing the proof of Theorem \ref{thm:mainA}}
We have built up to the following statement.
\begin{thm} \label{generalthm}
    Let $f:\RS\rightarrow \RS$ be a parabolic rational map of degree $d\geq 2$ and $\Omega = \{w_{1},\dots,w_{k}\}$ be the set of its rationally indifferent periodic points and suppose that all of the periodic points are fixed points. Let $\alpha$ be chosen as at the start of Section 3.2.  If $\varphi$  has the Bowen property on $\cG(\eta)$ for every $\eta>0$,  where $\cG(\eta)$ is as in \eqref{good} with $E(\alpha)=J\setminus \overline{B(\Omega, 2\alpha)}$ and $\lambda=\chi_{E}$, and $A(\Omega, \varphi)< P(J(f),\varphi)$, then $\varphi$  has a unique equilibrium state w.r.t. the dynamical system $(J(f),f)$ and it is fully supported on $J(f)$.
\end{thm}
\begin{proof}
We apply Theorem \ref{thm: CT16}.  Given $f$ be a parabolic rational map $f:\RS\rightarrow \RS$ of degree $d\geq 2$. By Theorem \ref{thm:expansive}, we know $f$ is positively expansive. By Lemma \ref{parabolicSpecification}, $f$ satisfies the specification property on $D(\alpha)$. For any $(x,n)\in \cG(\eta)$, choosing $k=1$, shows that  $\lambda(f^{n-1}x)\geq \eta$. Thus, for $\alpha>0$, $D(\alpha) \supseteq G(\eta)$. By Proposition \ref{pressureEstimate} and Lemma \ref{average}, the pressure gap condition is satisfied when $A(\Omega, \varphi)< P(\varphi)$.
\end{proof}
By Lemma \ref{HoldertoBowen}, any $\varphi:J(f) \to \bbR$ which is H\"older satisfies the `Bowen on good' hypothesis. To conclude the proof of Theorem \ref{thm:mainA}, all that remains is to remove our assumption (made only for convenience) that all the points in $\Omega$ are fixed points. If we did not assume that $\Omega(f)$ contains only fixed points, we can let $g=f^k$ be an iterate for which $\Omega(g)$ contains only fixed points, and not that $\Omega(f)=\Omega(g)$. Assume that  $A(\Omega, \varphi; f)= P(\Omega, \varphi; f) < P(\varphi;f)$. Then $A(\Omega, S_k \varphi; f^k) = kP(\Omega, \varphi; f) < kP(\varphi;f) = P(S_k\varphi; f^k)$. Thus we can apply Theorem \ref{generalthm} to obtain a unique equilibrium state for $(S_k\varphi; f^k)$. It follows from \S \ref{tfbackground} that we have a unique equilibrium state for $(\varphi, f)$.

\section{Family of geometric potentials} \label{geo.potential}
As an application of Theorem \ref{thm:mainA}, we consider the family of geometric potentials 
\[
\varphi_t(z) \coloneqq -t\log|f'(z)|.
\] 
Since there are no critical points in $J(f)$, the functions $\varphi_t: J(f) \to \bbR$ are well-defined and H\"older continuous. Note that the definition of $\varphi$ depends on the choice of metric. If we define the geometric potential using a different metric, it follows from \eqref{derivativeformula} that we differ by a coboundary, so the pressure and the set of equilibrium states are the same. Let $h = \text{HD}(J(f))$ denote the Hausdorff dimension of $J(f)$. Our claim is that $\varphi_t$ has a unique equilbrium state for $t<h$. This is not a new result, and we discuss other ways to obtain it below, but it is instructive that it follows easily from our criterion.
\begin{proof}[Proof of Corollary \ref{cor}]

Since $\varphi_t (z)\coloneqq -t\log|f'(z)|$ is H\"older continuous on $J(f)$, it thus has the Bowen property on $\cG(\eta)$. By Corollary 16 in \cite{DU91}, $P(\varphi_{t})>0$ for all $t< h$, where $h$ is the Hausdorff dimension of $J(f)$. Since $|(f^n)'(\omega)| =1$ for all $\omega\in \Omega$ with $f^n\omega = \omega$, we have $A(\Omega, \varphi) = 0$. Thus, $A(\varphi_{t})<P(\varphi_{t})$ and by Theorem A, there exists a unique equilibrium state for $\varphi_{t}$.
\end{proof}
\begin{remark} 
The proof can be obtained another way using existing results by checking that $\varphi_t$ is hyperbolic for $t<h$. Note that hyperbolicity is not immediate since we do not have a sign on $\sup_J \frac1n S_n \varphi_t$ in the Euclidean metric. For example, there are parabolic rational maps for which there exists $z\in f^{-1} \Omega$ with $|f'(z)|<1$ and $f(z)$ a fixed point. We see that $S_n \varphi_t(z) > 0$ for all $n$. 
Thus we cannot simply use the fact that $P(\varphi_t) > 0$ for all $t < h$ to conclude that $\varphi_t$ is hyperbolic for $t<h$.
Nevertheless, hyperbolicity can easily be verified from the existing literature. Przytycki showed in \cite{fP93} that any invariant measure has non-negative Lyapunov exponent, see \cite[Appendix A]{RL20} for a short recent proof.  Thus for $t >0$ and any measure $\mu$, we have $-t\int  \log|f’| d \mu \leq 0$. A potential is hyperbolic if and only if $\sup \int \varphi d \mu < P(\varphi)$, see \cite[Proposition 3.1]{IR}.  If $t<h$, we thus have $\sup \int \varphi_t d \mu \leq 0 <P(\varphi_t)$ and it follows that $\varphi_t$ is hyperbolic. 
\end{remark}
\begin{remark}
Another approach would be to use the existence of a Markov partition \cite{dR82, Aaronson} for a parabolic rational map and recode to a countable-state topological Markov chain with good distortion properties.  This approach was taken when the analogue of Corollary \ref{cor} was proved for the Manneville-Pomeau map by Sarig in \cite[Claim 2]{oS01}. We would not be surprised if this approach has been developed for parabolic rational maps previously, but we could not find it in the literature; the paper \cite{Aaronson} develops related tools but does not investigate equilibrium states. We also mention that Przytycki and Rivera-Letelier  \cite{PRL} used inducing schemes to study the potentials $-t\log|f'(z)|$ for an interval of values around $0$ for a wide class of non-uniformly hyperbolic rational maps allowing critical points in the Julia set, although they exclude parabolics except for quadratic map with real coefficients. 
\end{remark}
\begin{figure}[h!]
    \centering
   \begin{tikzpicture}[domain=-.5:2.5, scale=1.50]
  \draw[->] (-1,0) -- (3,0) node[right] {$t$};
  \draw[->] (0,-1.2) -- (0,2) node[above] {$P(\phi_t)$};
  \draw[color=blue,very thick,domain=-.5:1,<-]   plot (\x,{1.5*e^(-\x)-1.5*e^-1})    node[right] {};
   \draw[color=blue,very thick,domain=1:2.5,->]   plot (\x,{0})    node[right] {};
  \filldraw[blue] (1,0) circle (2pt) node[below]{$h$};
  \filldraw[blue] (0,.948) circle (2pt) node[above right]{$h_{\text{top}}(f)=\log d$};
\end{tikzpicture}
    \caption{Pressure of the potential $\phi_{t}(x)=-t\log|f'(x)|$ for a parabolic rational map, $f$. The Hausdorff dimension of $J(f)$ is $h$.}
    \label{fig:PressureMP}
\end{figure}
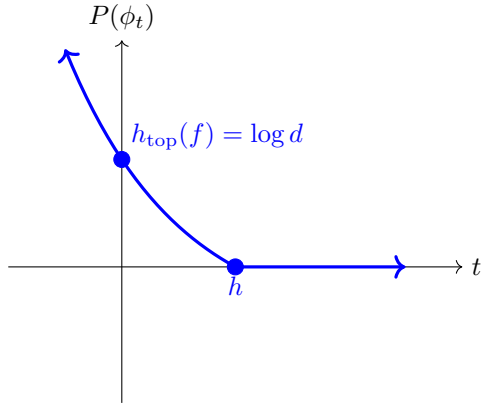

By results of Walters \cite{pW92}, uniqueness of equilibrium states for $(-\infty, h)$ holds if and only if the pressure function $t \to P(\varphi_t)$ is $C^1$ on this interval. There is a phase transition at $t=h$ where uniqueness of the equilibrium state and smoothness of the pressure function fails. Our argument thus gives a simple justification for the famous picture of a non-uniformly hyperbolic pressure function in this setting. We remark that Haydn and Isola \cite{HI} showed by hard analysis of the spectrum of the transfer operator that the pressure function is analytic for $t\in[0,h)$ which thus provides the earliest indirect proof that we know of for our corollary.

\section{Positive entropy and the hyperbolicity criterion} \label{pos.entropy}
We now apply a method by Call and the second-named author \cite{CallT22} to show that our equilibrium states have positive entropy.  These techniques were developed further by Call and the following result is a version of \cite[Theorem B]{Call22}. The non-invertible case follows the same proof, with some simpler bookkeeping. The statement we give here is also simpler because we assume a global positive expansivity condition.
\begin{thm}[Call] \label{thm: K}
Let $X$ be a compact metric space, $T:X\rightarrow X$ a positively expansive continuous map, and $\varphi:X\rightarrow \mathbb{R}$ a continuous potential function.  
Assume that $(X, f)$ admits a one-sided $\lambda$-decomposition with the following properties:
\begin{enumerate}
    \item For all $\eta > 0$, the set $\mathcal G(\eta)$ has specification;
    \item For all $\eta > 0$, the potential $\varphi$ has the Bowen property on
    $\mathcal G(\eta)$.
\end{enumerate}
Assume furthermore that
\begin{equation} \label{prodpressure}
P\Bigl(
\bigcap_{n \in \mathbb{N} \cup \{0\}} (f^{-n} \times f^{-n})\,\widetilde{\lambda}^{-1}(0),
\Phi
\Bigr)
< 2P(\varphi; f) = P(\Phi, f \times f),
\end{equation}
where $\Phi \colon X \times X \to \mathbb{R}$ is defined by $\Phi(x,y) = \varphi(x) + \varphi(y)$,
and $\widetilde{\lambda} \colon X \times X \to \mathbb{R}$ is given by $\widetilde{\lambda}(x,y) = \lambda(x)\lambda(y)$. Then $(X,f,\varphi)$ has a unique equilibrium state, and this equilibrium state
has the $K$--property.
\end{thm}
The method of proof is simply that the criteria guarantee a unique equilibrium state for $(X \times X, f \times f, \Phi)$. Ledrappier showed that a unique equilibrium state for $(X \times X, f \times f, \Phi)$ guarantees a unique equilibrium state with the $K$-property (i.e. every non-trivial factor of the measure has positive entropy) for the system $(X, f, \varphi)$, see \cite[Theorem 3.3]{Call22}. To show that pressure gap holds for the product system, we will apply the following result by Call.
\begin{thm}[\cite{Call22},Theorem 4.8]\label{thm: productGap}
Let $\mu$ be the unique equilibrium state of $(X,f,\varphi)$, and suppose that $\mu(\lambda^{-1}(0))<\frac{1}{2}$. If the entropy map is upper semicontinuous on the product system (in particular if $(X, f)$, and thus the product system, is positively expansive), then 
\[
P\Bigl(
\bigcap_{n \in \mathbb{N}\cup\{0\}} (f^{-n} \times f^{-n})\,\widetilde{\lambda}^{-1}(0),
\Phi
\Bigr)
< 2P(\varphi; f).
\]
\end{thm}
The following result proves Theorem \ref{thm:posentropy} since measures with the $K$-property have positive entropy.
\begin{thm} 
Suppose we are in the setting of Theorem \ref{thm:mainA}, then
the unique equilibrium state has the $K$-property. 
\end{thm}
\begin{proof}
 We verify the hypotheses of Theorem \ref{thm: K}. In \S \ref{mainresults}, for any $\beta>0$ small, we defined a function $\lambda^{\beta}$ and a decomposition $(\cG^\beta(\eta), \cB^\beta (\eta))$. We write $\alpha$ for the fixed choice of scale we used in  \S \ref{mainresults}.  Any $\cG^\beta(\eta)$ with $\beta< \alpha$ satisfies hypotheses (1) and (2), so all that is left to show is the pressure gap \eqref{prodpressure} for the product system. 
 Let $\mu$ be the unique equilibrium state for $(X,f,\varphi)$ provided by Theorem \ref{thm:mainA}. Since $\mu$ is ergodic and an equilibrium state, since $A(\Omega, \varphi)< P(J(f),\varphi)$, we must have $\mu(\Omega)=0$. Since $\bigcap_{\beta>0} B(\Omega, 2\beta) = \Omega$, we can choose  $\beta>0$ sufficiently small so that $\mu(B(\Omega, 2\beta)) <\frac{1}{2}$. We have $(\lambda^{\beta})^{-1}(0) = B(\Omega, 2\beta)$, and thus for this choice of $\beta$, we have $\mu((\lambda^{\beta})^{-1}(0))<\frac{1}{2}$. Thus, by Theorem \ref{thm: productGap}, \eqref{prodpressure} holds for the $\lambda^\beta$-decomposition. We have met the hypotheses of Theorem \ref{thm: K} and we can conclude that the unique equilibrium state $\mu$ for $(J(f),f,\varphi)$ has the $K$-property.
\end{proof}
The condition that potential $\varphi$ is hyperbolic in the sense that there exists $n$ so that  $\sup_J \frac{1}{n}S_n \varphi  < P(\varphi)$ is equivalent to the condition that every equilibrium state of $\varphi$ has positive entropy\footnote{The main result of Inoquio-Renteria and Rivera-Letelier \cite{IR}, which we do not need and relies on $f$ being a rational map, is that the hyperbolicity condition for $\varphi$ is equivalent to the condition that the Lyapunov exponent of each equilibrium state of $f$ for the potential $\varphi$ is strictly positive.}.  See Proposition 3.1 of \cite{IR} for a short general proof of this fact. Since our unique equilibrium state has positive entropy, we conclude that $\varphi$ is hyperbolic.


We now have everything we need to prove Corollary \ref{cor2}. Let $f:\RS\rightarrow \RS$ be a parabolic rational map of degree $d\geq 2$ and $\Omega$ be its set of rationally indifferent periodic points. Let $\varphi: J(f) \rightarrow \bbR$ be a H\"older potential. If $\varphi$ is hyperbolic, then it is obvious that $A(\Omega, \varphi)< P(\varphi)$. If  $A(\Omega, \varphi)< P(\varphi)$, then by our main theorems we have obtained a unique equilibrium state, and it has positive entropy. We conclude that $\varphi$ is hyperbolic. By 
\cite[Proposition 3.1]{IR}, there is a potential $\tilde \varphi$ cohomologous to $\varphi$ (and thus with the same set of equilibrium states) so that $\sup_J \tilde \varphi < P(\tilde \varphi)$. We are now in the setting of Szostakiewicz, Urbanski and Zdunik \cite{SUZ}. The unique equilibrium state for $\tilde \varphi$ (and thus $\varphi$) satisfies all the strong statistical properties provided by \cite{SUZ}, which include exponential decay of correlations, the central limit theorem and the law of iterated logarithms.

\subsection*{Acknowledgements} We thank Roland Roeder, Liz Vivas and Juan Rivera-Letelier for helpful discussions and comments. 
\bibliographystyle{siam}
\bibliography{main.bib}

\end{document}